\documentclass[reqno, 12pt]{amsart}

\usepackage{mathrsfs}
\usepackage{amscd}
\usepackage{amsmath}
\usepackage{latexsym}
\usepackage{amsfonts}
\usepackage{amssymb}
\usepackage{amsthm}
\usepackage{graphicx}
\usepackage{hyperref}
\usepackage{makecell}
\usepackage{array}
\usepackage{booktabs}
\usepackage{multirow}
\usepackage{color,xcolor}

\newtheorem{theorem}{Theorem}[section]

\newtheorem{lemma}[theorem]{Lemma}
\newtheorem{definition}[theorem]{Definition}

\newtheorem{remark}[theorem]{Remark}

\numberwithin{equation}{section}

\title[Inverse random source problem]{Stability for an inverse random source problem of the biharmonic Schr\"odinger equation}

 {\author[T. Wang]{Tianjiao Wang}\address{School of Mathematical Sciences, Zhejiang University, Hangzhou 310058, China}}\email{wangtianjiao@zju.edu.cn}

\author[X. Xu]{Xiang Xu}
\address{School of Mathematical Sciences, Zhejiang University, Hangzhou 310058, China}
\email{xxu@zju.edu.cn}
\author[Y. Zhao]{Yue Zhao}
\address{School of Mathematics and Statistics, and Key Lab NAA-MOE, Central China Normal University,
Wuhan 430079, China}
\email{zhaoyueccnu@163.com}


\begin{document}

\begin{abstract}
In this paper, we study the inverse random source scattering problem for the biharmonic Schr\"odinger equation in two and three dimensions. The driven source is assumed 
to be a generalized microlocally isotropic Gaussian random function whose covariance operator is a classical  pseudodifferential operator. 
We examine the meromorphic continuation and estimates for the resolvent of the biharmonic Schr\"odinger operator with respect to complex frequencies, which yields the well-posedness of the direct problem and a Born series expansion for the solution.
For the inverse problem, we present quantitative stability estimates in determining the micro-correlation strength of the random source by using the correlation far-field data at multiple frequencies. The key ingredient in the analysis is employing scattering theory to obtain the analytic domain and an upper bound for the resolvent of the biharmonic Schr\"odinger operator and applying a quantitative analytic continuation principle in complex theory. The analysis only requires data in a finite
interval of frequencies.

\end{abstract}

\maketitle

\section{Introduction}

Inverse source scattering problem aims at determining the source from appropriate measurements away from its support. Such kind of inverse problems arise in many scientific and industrial areas including quantum physics, biomedical imaging, radar and sonar, and geophysical exploration \cite{isakov}. 
In situations involving unpredictable environments, incomplete knowledge of the system, and random measurement noise,
it is usually assumed that the source is created by random process. As a consequence, the scattered field is also random. Compared with their deterministic counterparts, stochastic inverse problems are more difficult to handle due to randomness and uncertainties. In these situations, it is less meaningful to characterize the source by a particular realization but to determine their statistical properties. We refer the reader to \cite{bao2016random, bao2017random, liang2024stability, li2022far} and references therein for some recent advancements on inverse random source problems in acoustic, elastic and electromagnetic waves.

Recently, scattering problems for biharmonic waves arising from thin plate elasticity have received much attention due to their important applications in 
the control of destructive surface waves and the construction of ultrabroadband elastic cloaking devices \cite{3, 4, 5}. 
For the inverse random source problem modeled by an additive white noise, a regularized Kaczmarz method  was developed in \cite{LW_1} to reconstruct the mean and variance of the random source by multi-frequency data.
Apart from the white noise, another important random source function was considered in \cite{li2022far}, which is called the generalized microlocally isotropic Gaussian (GMIG) random field. GMIG is a generalized random function whose covariance operator is assumed to be a classical pseudodifferential operator with a special form of principle symbol introduced by \cite{lassas2008inverse}. For biharmonic wave equation in the absence of potentials, it was proved  in \cite{li2022far} that the principle symbol of the covariance operator could be uniquely recovered by the far-field correlation data measured at all high frequencies.

Driven by the signification applications, we consider an inverse random source scattering problem for the biharmonic Schr\"odinger equation
\begin{align}\label{DP}
		(\Delta^2 - k^4 + V)u=f &\quad \text{in}\quad \mathbb{R}^d,
\end{align} 
where $d=2, 3$, and $k>0$ is the wavenumber. $V$ is the bounded potential function and is assumed to have a compact support.
The scattered field $u$ and its Laplacian $\Delta u$ are required to satisfy the Sommerfeld radiation condition
\begin{align}\label{src}
 \lim\limits_{r \to \infty} r^{\frac{d-1}{2}} (\partial_r u- {\rm i}k u)=0, \quad \lim\limits_{r \to \infty} r^{\frac{d-1}{2}} (\partial_r (\Delta u)- {\rm i}k (\Delta u))=0 &\quad r=|x|.
\end{align}
The source $f$ is assumed to be a generalized microlocally isotropic Gaussian random field of order $-m$ such that its covariance operator is a classical pseudodifferential 
operator with principle symbol $h(x)|\xi|^{-m}$, where $h$ is called the micro-correlation strength of the random source. 

Although uniqueness and numerical reconstruction methods have been extensively investigated for inverse random source scattering problems, little is known about the corresponding stability estimates. 
A logarithmic stability was obtained in \cite{liang2024stability} for the inverse random source problem of the Helmholtz equation in inhomogeneous medium driven by 
a white noise. However, the proof in \cite{liang2024stability} relies on the It$\acute{\rm o}$ isometry of the white noise, which might not be available for GMIG random fields.
In recent works \cite{RSWXZ, WXZ_HN}, stability estimates for inverse random source problems driven by GMIG random fields were proved for acoustic and elastic waves.
Motivated by \cite{WXZ_HN}, the goal of this work is to derive quantitative stability estimates for the inverse random source problem driven by GMIG random field for the fourth-order biharmonic Schr\"odinger equation in two and three dimensions. Compared with second-order wave equations considered in \cite{WXZ_HN}, due to the increase of the order, the analysis for the biharmonic operator is much more sophisticated since the properties of the solutions for the higher-order equations are more complicated.
Another issue is that the Green function of the biharmonic wave equation in two dimensions involves Hankel function, which requires more delicate analysis than the exponential one in three dimensions.

We examine both the direct and inverse random source problems for the biharmonic Schr\"odinger equation in two and three dimensions. For the direct problem, we investigate the meromorphic continuation of the resolvent of the biharmonic Schr\"odinger operator, which yields an analytic domain and estimates for the resolvent with respect to complex frequencies. It should be noticed that for $m\leq d$, the potential belongs to Sobolev space with a  negative smoothness index, which is too rough to exist pointwisely and should be understood as a distribution. Hence, the direct problem requires a careful analysis and the classical scattering theory in \cite{DZ} cannot be applied directly.
The well-posedness of the direct problem follows from the analysis of the resolvent. Moreover, we show that the solution admits a Born series expansion, which
is useful in the study of the inverse problem. 

As for the inverse problem, we utilize multi-frequency far-field pattern to recover the strength of the random source. To achieve this goal,
we expand the far-field pattern into Born series and establish a relation between the strength of the random source and the correlation of the zeroth-order term in the Born series of the far-field pattern. To handle the higher-order terms, we employ resolvent estimates to obtain decay estimates at high frequencies. Then the stability follows from an application of quantitative unique continuation in complex theory. More importantly, we also present another stability estimate by using far-field pattern averaged over
the high-frequency band, which quantifies a uniqueness result implied by \cite{li2022far}. Such type of averaged data was widely employed to prove uniqueness results in inverse random scattering problems, see e.g. \cite{CARO, lassas2008inverse, li2022far}.
To handle this case, we conduct more sophisticated analysis for the first-order term in Born series
expansion of the far-field pattern, instead of applying the resolvent estimate directly. Our method only requires data in a finite
interval of frequencies, while uniqueness results in \cite{CARO, lassas2008inverse, li2022far, LW_2} require data measured 
at all high frequencies. 

This paper is organized as follows. In Section \ref{preli}, we introduce a mathematical description of the generalized
Gaussian random field and some standard Sobolev spaces. Section \ref{mero} is devoted to the study of the meromorphic continuation of 
the resolvent for the biharmonic Schr\"odinger operator. An analytic domain and resolvent estimates are achieved, which lead to the well-posedness of the 
direct problem and a Born series expansion of the solution. Section \ref{ip} addresses stability estimates for the inverse random source problem.

\section{Preliminaries}\label{preli}

In this section, we introduce the generalized microlocally isotropic Gaussian random field. Denote the space of test functions by $\mathcal{D}$, which is $C_0^\infty(\mathbb R^d;\mathbb R)$ equipped with a locally convex topology.
Denote the dual of $\mathcal{D}$ by $\mathcal{D}'$ which is the generalized function space.
Given a complete probability space $(\Omega, \mathcal{A}, \mathbb P)$, a real-valued random field $g : \Omega \to \mathcal{D}'$ is called a generalized microlocally isotropic Gaussian random field if $\omega \mapsto \langle g(\omega), {\phi} \rangle$ is a Gaussian random variable for all $\phi \in \mathcal{D}$. Obviously, $g(\omega)$ can be extended to the
space of complex-valued smooth functions with compact support as \[
\langle g(\omega), {\phi} \rangle= \langle g(\omega), \Re{\phi} \rangle + {\rm i }\langle g(\omega), \Im{\phi} \rangle, \quad \phi \in C_0^\infty(\mathbb R^d; \mathbb C).
\] The expectation $\mathbb E  g \in \mathcal{D}'$ and covariance $\mathrm{Cov}\, g : \mathcal{D} \times \mathcal{D}: \to \mathbb{R} $ are respectively defined by \[
\mathbb{E}g\, :\, \phi \mapsto \mathbb{E}\langle g, {\phi}\rangle, \quad {\phi} \in \mathcal{D},
\]\[
\mathrm{Cov}\,g\,:\, (\phi_1, \phi_2) \mapsto \mathrm{Cov}(\langle g,\phi_1 \rangle, \langle g,\phi_2 \rangle), \quad \phi_1, \phi_2 \in \mathcal{D},
\] where \[
\mathrm{Cov}(\langle g,\phi_1 \rangle, \langle g,\phi_2 \rangle)=\mathbb E [(\langle g,\phi_1 \rangle- \mathbb E \langle g,\phi_1 \rangle)(\langle g,\phi_2 \rangle- \mathbb E \langle g,\phi_2 \rangle)].
\] Define the covariance operator $C_{g}$ by \[\langle C_{g} \phi_1, \phi_2 \rangle = \mathrm{Cov}(\langle g,\phi_1 \rangle, \langle g,\phi_2 \rangle), \quad \phi_1, \phi_2 \in \mathcal{D}.\] Let $K_{g}(x,y)$ be the Schwarz kernel of $C_{g}$, i.e., \[
\langle C_{g} \phi_1, \phi_2 \rangle=\int_{\mathbb{R}^d}\int_{\mathbb{R}^d} \phi_2(y) K_{g}(x,y) \phi_1(x) \,\mathrm{d}x\mathrm{d}y.
\]  

We consider the generalized microlocally isotropic Gaussian random fields with zero expectation defined as follows. \begin{definition}\label{Def2.1}
	A generalized Gaussian random field $g$ with zero expectation is said to be microlocally isotropic of order $-m \in \mathbb{R}$ in $D \subset \mathbb{R}^d, d=2, 3,$ if its covariance operator $C_{g}$ is a classical pseudodifferential operator of order $-m$ with the principal symbol $\mu(x) |\xi|^{-m}$, where $\mu(x) \in C^\infty_0(\mathbb{R}^d; \mathbb R)$ is a real-valued function with $\text{supp} \, \mu \subset \subset D$. Here $\mu$ is called the micro-correlation strength of $g$.
\end{definition} Let $c_{g}$ be the symbol of the operator $C_{g}$, i.e., \[
C_{g}(\phi)(x)=(2 \pi)^{-d} \int_{\mathbb{R}^d} e^{{\rm i} x \cdot \xi} c_{g}(x, \xi) \hat{\phi}(\xi) \,\mathrm{d}\xi, \quad \phi \in \mathcal{D},
\] which gives \begin{equation*}
	K_{g}(x,y)=(2\pi)^{-d} \int_{\mathbb{R}^d} e^{{\rm i}(x-y) \cdot \xi} c_{g}(x, \xi) \,\mathrm{d}\xi.
\end{equation*}
The following lemma proved in \cite{CARO} concerns the regularity of the GMIG random field.
\begin{lemma}\label{Lem2.1}
	Let $g$ be a GMIG random field of order $-m$ in $D$. Then $g \in W^{\frac{m-d}{2}-\delta,p}(D)$ almost surely for all $\delta>0$ and $1<p<\infty$.
\end{lemma}

We also introduce some function spaces.
Denote the standard Sobolev spaces by $W^{s,p}:=W^{s,p}(\mathbb{R}^d)$ and denote $H^s:=W^{s,2}$. The function spaces $H^s_{\rm loc}:=H^s_{\rm loc}(\mathbb{R}^d)$ and $W^{s,p}_{\rm comp}:=H^s_{\rm comp}(\mathbb{R}^d)$ are respectively defined by \begin{align*}
	H^s_{\rm loc}(\mathbb{R}^d) &:=\{u \in \mathcal{D}'(\mathbb{R}^d):\chi u \in H^s(\mathbb{R}^d),\forall \, \chi \in C_0^\infty(\mathbb{R}^d)\}, \\ H^s_{\rm comp}(\mathbb{R}^d) &:=\{u \in H^s(\mathbb{R}^d):\exists \, \chi \in C_0^\infty(\mathbb{R}^d),\,\chi u =u\}.
\end{align*} The bounded linear operators between two function spaces $X$ and $Y$ are denoted by $\mathcal{L}(X,Y)$ with the operator norm $\|\cdot\|_{\mathcal{L}(X,Y)}$.  Moreover, we use $a \lesssim b $ to stand for $a \le C b$, where $C>0$ is a generic constant whose specific value is not required but should be clear from the context.

\section{Meromorphic continuation and estimates for the resolvent}\label{mero}

In this section, we investigate the meromorphic continuation and resolvent estimates for the resolvent $R_V(\lambda) = (\Delta^2 - \lambda^4 + V)^{-1}$ of the biharmonic Schr\"odinger
operator in two and three dimensions.

We begin with the free resolvent $R_0(\lambda) = (\Delta^2 - \lambda^4)^{-1}$ of the biharmonic operator. Let $\widetilde R_0(\lambda) = (-\Delta - \lambda^2)^{-1}$
be the free resolvent of the Laplacian operator, whose kernel $\Phi_\lambda(x,y)$ has the explicit form
\begin{align*}
    \Phi_\lambda(x,y)=\begin{cases}
    \frac{{\rm i}}{4}H_0^{(1)}(\lambda|x-y|), & \quad d=2,\\[5pt]
        \frac{e^{{\rm i}\lambda|x-y|}}{4\pi|x-y|}, &\quad d=3.
    \end{cases}
\end{align*} 
According to the decomposition formula
\[
(\Delta^2 - \lambda^4)^{-1} = \frac{1}{2\lambda^2} \Big((-\Delta-\lambda^2)^{-1} - (-\Delta+\lambda^2)^{-1}\Big),
\]
the kernel of $R_0(\lambda)$ is
\[
G_\lambda(x,y)=\frac{1}{2\lambda^2}(\Phi_\lambda(x,y)-\Phi_{{\rm i}\lambda}(x,y)).
\] 
Hence, given $f\in L_{\rm comp}^2(\mathbb R^d)$, the resolvents of the free Laplacian operator and biharmonic operator have the representations
\[
\widetilde R_0(\lambda)(f)(x):=\int_{\mathbb R^d} \Phi_\lambda (x,y) f(y)\,\mathrm{d}y
\] 
and
\[
R_0(\lambda)(f)(x):=\int_{\mathbb R^d} G_\lambda (x,y) f(y)\,\mathrm{d}y=\frac{1}{2\lambda^2} \Big(\widetilde R_0(\lambda)(f)(x)-\widetilde R_0({\rm i}\lambda)(f)(x)\Big).
\]

In three dimensions, with the help of strong Huygens principle and the spectral theorem, we have the following analyticity and estimates for the free resolvent  
$\widetilde{R}_0(\lambda)$. The proof adapts the arguments in 
	\cite{DZ}. 
\begin{lemma}\label{lem1}
 The free resolvent $\widetilde{R}_0(\lambda)$ is analytic for $\lambda \in \mathbb{C} $ with $\Im \lambda>0$ as a family of operators\[
 \widetilde{R}_0(\lambda): H^s(\mathbb{R}^3) \to H^t(\mathbb{R}^3)
 \] for all $s \in \mathbb{R}$ and $t \in [s,s+2]$. Furthermore, given a cut-off function $\chi \in C_0^\infty(\mathbb{R}^3)$, $\widetilde{R}_0(\lambda)$ can be extended to an analytic family of operators for $\lambda \in \mathbb{C}$ as follows \[
 \chi \widetilde{R}_0(\lambda) \chi : H^s(\mathbb{R}^3) \to H^t(\mathbb{R}^3)
 \] with the resolvent estimates 
 \begin{equation}\label{res1}
     \|\chi \widetilde{R}_0(\lambda) \chi \|_{\mathcal{L}(H^s,H^t)} \lesssim (1+|\lambda|)^{t-s-1}e^{L (\Im \lambda)_-},
 \end{equation} where $s \in\mathbb R$ and $t\in [s,s+2]$, $v_-:=\max\{-v,0\}$ and $L$ is a constant larger than the diameter of the support of $\chi$.
 \end{lemma}
 
 \begin{proof}
From the functional calculus in \cite[Appendix B.2]{DZ} we denote \[
 U(t):=\frac{\sin t\sqrt{-\Delta}}{\sqrt{-\Delta}}, 
 \] which yields a representation of $\widetilde{R}_0(\lambda)$ as follows
 \[
 \widetilde{R}_0(\lambda)=\int_0^\infty e^{{\rm i}\lambda t} U(t)\,\mathrm{d}t.
 \] By the spectral representation in \cite[Appendix B.1]{DZ} we have the mapping properties 
\[
\partial_t^n U(t) \,:\, H^s(\mathbb{R}^3) \to H^{s-n+1}(\mathbb{R}^3),\quad n \in \mathbb{N},\quad s \in \mathbb{R}.
\] Therefore, for $\Im \lambda>0$ the family of operators \[
\widetilde{R}_0(\lambda): H^s(\mathbb R^3) \to H^s(\mathbb R^3)
\] is analytic with respective to $\lambda$. For $f \in H^s$, we denote $u(\lambda)=\widetilde{R}_0(\lambda)f$ which gives $\Delta u(\lambda)=-\lambda^2u(\lambda) -f \in H^s$. Hence, both $u(\lambda)$ and $\Delta u (\lambda)$ are analytic families of functions in $H^s$ when $\Im \lambda>0$. For $\phi \in H^{s+2}$, noting \begin{align*}
\langle u(\lambda), \phi \rangle_{H^{s+2}}&=\langle (I-\Delta) u(\lambda), (I-\Delta) \phi \rangle_{H^{s}}\\
&=\langle \Delta u(\lambda), \Delta \phi \rangle_{H^{s}}+\langle u(\lambda), \phi \rangle_{H^{s}}-2\langle \Delta u(\lambda), \phi \rangle_{H^{s}},
\end{align*}
we have for $\lambda$ 
with $\Im \lambda>0$ that the function $\langle u(\lambda), \phi \rangle_{H^{s+2}}$
is analytic with respect to $\lambda$. Thus, $u(\lambda)$ is an analytic family of functions in $H^{s+2}$, which yields that $\widetilde R_0(\lambda)$ is also an analytic family of operators for $\Im\lambda>0$ in the following sense \[
\widetilde R_0(\lambda): H^s \to H^{s+2}.
\]

We next prove the resolvent estimates and the analycity of $\chi \widetilde R_0(\lambda) \chi $ for $\lambda \in \mathbb C$. The strong Huygens's principle implies that 
\[
 (U(t) \chi)(x)=0,\quad \text{when} \,\, t>\sup\{|x-y|:y \in \text{supp} \chi\}.
 \]  Then we obtain \[
\chi  \widetilde{R}_0(\lambda) \chi =\int_0^L e^{{\rm i}\lambda t} \chi U(t) \chi \,\mathrm{d}t,
\] where $L> \text{diam}\,\text{supp} \chi$.
  Furthermore, we have \[
\chi  \widetilde{R}_0(\lambda) \chi=({\rm i}\lambda)^{-1}\int_0^L  \partial_t(e^{{\rm i}\lambda t}) \chi U(t) \chi \,\mathrm{d}t=-({\rm i}\lambda)^{-1}\int_0^L  e^{{\rm i}\lambda t} \chi  \partial_t(U(t)) \chi \,\mathrm{d}t.
\] It follows from the spectral representation that $U(t)=\mathcal O_{\mathcal{L}(H^s, H^{s+1})}(1)$,  $\partial_tU(t)=\cos t{\sqrt{-\Delta}}=\mathcal O_{\mathcal{L}(H^s, H^s)}(1)$. Hence, we know when $s=t$ or $s=t+1$ that \[
\chi\widetilde{R}_0(\lambda)\chi : H^s  \to H^t
\] is an analytic family of operators for $\lambda \in \mathbb C$ and the resolvent estimates \eqref{res1} hold. For the case of $t=s+2$, note that \[
\Delta(\chi \widetilde{R}_0(\lambda) \chi)=\chi \Delta \widetilde{R}_0(\lambda) \chi+[\Delta, \chi] \widetilde{R}_0(\lambda) \chi= (-1-\lambda^2)\chi \widetilde{R}_0(\lambda)\chi + [\Delta, \chi] \widetilde{R}_0(\lambda) \chi,
\] where $[\Delta, \chi]: =\Delta \chi -\chi \Delta$ is a first-order differential operator. Then the above identities together with the analyticity and resolvent estimates for $t=s, s+1$ imply the corresponding results for the case $t=s+2$.
The remaining cases follow by applying the interpolation inequality. The proof is now complete.
 \end{proof}

 For two-dimensional case, the strong Huygens principle is no longer available. To investigate the resolvent, we consider  the integral representation of the Hankel function \cite{Finco}
 \begin{equation} \label{Hankel}
		H_0^{(1)}(z) = C e^{{\rm i}z} \int_{0}^\infty e^{-s}s^{-1/2}(s/2-{\rm i}z)^{-1/2}\,\mathrm{d}s,
	\end{equation}
 where $C$ is a constant, and the branch cut of $\sqrt{z}$ is chosen such that  $\sqrt{z}$ is analytic for $z \in \mathbb C \backslash (-\infty,0]$ and $\sqrt{1}=1$. Thus, for $\lambda \in \mathcal S :=\mathbb C \backslash (-\infty,0]{\rm i}$ we have $\sqrt{s/2-{\rm i}\lambda|x-y|} \in \mathbb C \backslash (-\infty,0]$ for any $x,y \in \mathbb R^2$ and $s>0$. 
Moreover, from \eqref{Hankel} we have the estimate \begin{align}\label{Hankelesti}
     |\Phi_\lambda(x,y)| \lesssim \frac{e^{(\Im \lambda)_-|x-y|}}{\sqrt{\alpha(\lambda)|x-y|}} \int_0^\infty e^{-t} t^{-\frac{1}{2}}\,\mathrm{d}t \lesssim \frac{e^{(\Im \lambda)_-|x-y|}}{\sqrt{\alpha(\lambda)|x-y|}}, \quad \lambda \in \mathcal S,
 \end{align} 
 where $a_-:=\max\{-a,0\}$ and $\alpha(\lambda)$ is given by \begin{align*}
     \alpha(\lambda)=\begin{cases}
         |\Re \lambda|, & \Im \lambda \le 0,\\
         |\lambda|, & \Im \lambda > 0.
     \end{cases}
 \end{align*} Using \eqref{Hankelesti} and the Schur's Lemma, we derive the analyticity and estimates of the resolvent $\widetilde R_0(\lambda)$
 in the following lemma in two dimensions. 
 \begin{lemma}\label{lemma2}
	The free resolvent $\widetilde R_0(\lambda)$ is analytic for $\lambda \in \mathcal{S}$, $\Im \lambda>0$ as a family of operators \[
	\widetilde R_0(\lambda): H^s(\mathbb{R}^2) \to H^t(\mathbb{R}^2)
	\] for all $s \in \mathbb{R}$ and $t \in [s,s+2]$. Furthermore, given a cut-off function $\chi \in C_0^\infty(\mathbb{R}^2)$, $\widetilde{R}_0(\lambda)$ can be extended to 
	an analytic family of operators for $\lambda \in \mathcal{S}$ as follows \[
	\chi \widetilde R_0(\lambda) \chi : H^s(\mathbb{R}^2) \to H^s(\mathbb{R}^2)
	\] with the resolvent estimates 
	\begin{equation*}
		\|\chi \widetilde R_0(\lambda) \chi \|_{\mathcal{L}(H^s,H^t)} \lesssim \alpha(\lambda)^{-1/2}(1+|\lambda|)^{t-s}e^{L (\Im \lambda)_-},
	\end{equation*} where $s \in \mathbb R$ and $t\in [s,s+2]$, $t_-:=\max\{-t,0\}$ and the constant $L$ is only dependent on $supp \, \chi$.
\end{lemma}
\begin{proof}
    By Schur's lemma \cite[Appendix A.5]{DZ} and \eqref{Hankelesti} we have \begin{equation}\label{L2}
		\|\chi \widetilde R_0(\lambda) \chi\|_{L^2 \to L^2} \lesssim \alpha(\lambda)^{-\frac{1}{2}}e^{L(\Im \lambda)_-}
	\end{equation} with $L> \text{diam}\,\text{supp} \chi$.
	Then we have for $f \in H^1(\mathbb{R}^2)$
	\begin{align}
		\partial^\alpha (\chi \widetilde R_0(\lambda) \chi f) &=\partial^\alpha(\chi (\Phi_\lambda * (\chi f)))=(\partial^\alpha  \chi)(\Phi_\lambda * (\chi f))+\chi (\Phi_\lambda * \partial^\alpha (\chi f)) \notag\\ &=(\partial^\alpha  \chi)(\Phi_\lambda * (\chi f))+ \chi (\Phi_\lambda * \partial^\alpha \chi f)+\chi (\Phi_\lambda * \partial^\alpha f \chi ), \label{Cut}
	\end{align}  where $|\alpha|=1$. Combining \eqref{L2}--\eqref{Cut} gives the resolvent estimate for $s=t=1$, which also holds for all $s=t \in \mathbb{N}$ by induction. Further, using the interpolation of fractional order Sobolev spaces in \cite[7.57]{adams}, i.e., \[H^s(\Omega)=[L^2(\Omega),H^n(\Omega)]_{s/n}, \quad  n>s>0,\quad n \in \mathbb{N},\] we have that the resolvent estimates hold for all $s=t \ge 0$. 
	
	For the case $s=t<0$, using the definition of convolution of distributions and the symmetry of the fundamental solution, we obtain \begin{align*}
		\left\langle\chi \widetilde R_0(\lambda) \chi f, \phi\right\rangle &=\left\langle \Phi_\lambda * (\chi f),\chi \phi  \right\rangle=\left\langle \chi f, \left\langle\Phi_\lambda(x-y), \phi \chi \right\rangle_x \right\rangle_y \\ &=\left\langle \left\langle \chi f , \Phi_\lambda(y-x),\phi \chi  \right\rangle_x \right\rangle_y
		\\&= \left\langle \chi f , \Phi_\lambda * (\phi \chi) \right\rangle=\left\langle f,\chi \widetilde R_0(\lambda)\chi \phi \right\rangle, \quad \phi \in C_0^\infty(\mathbb{R}^2).
	\end{align*}
	Thus, we have that the resolvent $\chi R_0(\lambda) \chi: H^s\to H^s$ is a bounded operator with the corresponding resolvent estimate as in the lemma. Next, we consider the gradient of $\chi\widetilde R_0(\lambda)\chi f$. By direct calculations we have
	\begin{align*}
 \nabla(\chi\widetilde R_0(\lambda)\chi f) &=(\nabla\chi )\widetilde R_0(\lambda)\chi f+\chi \nabla(\widetilde R(\lambda)\chi f)=(\nabla\chi )\Phi_\lambda *(\chi f)+\chi \nabla(\Phi_\lambda * (\chi f)) \\ &=(\nabla\chi )\Phi_\lambda *(\chi f)+\chi (\nabla\Phi_\lambda * (\chi f)).
    \end{align*} From \eqref{Hankel}, there holds \begin{align*}
		  |\nabla_x \Phi_\lambda(x,y)| \lesssim \frac{|\lambda|e^{(\Im \lambda)_-|x-y|}}{|\alpha(\lambda)(x-y)|^{\frac{1}{2}}} + \frac{e^{(\Im \lambda)_-|x-y|}}{|\alpha(\lambda)|^\frac{1}{2}|x-y|^{\frac{3}{2}}}, \quad \lambda \in \mathcal S.
	\end{align*} Therefore, following an analogous argument as for the case $s=t$ we obtain \[
    \|\chi \widetilde { R}_0(\lambda)\chi\|_{H^s \to H^{s+1}} \lesssim |\alpha(\lambda)|^{-1/2}(1+|\lambda|)e^{L (\Im \lambda)_-}.
    \] As the following identity holds \begin{equation}\label{sec}
\Delta(\chi \widetilde R_0(\lambda) \chi )=\chi  \Delta \widetilde R_0(\lambda) \chi+[\Delta, \chi] \widetilde R_0(\lambda) \chi= (-1-\lambda^2)\chi \widetilde R_0(\lambda)\chi + [\Delta, \chi] \widetilde R_0(\lambda) \chi,
\end{equation} where $[\Delta, \chi]: =\Delta \chi -\chi \Delta $ is a first order differential operator, 
it follows from \eqref{sec} and the corresponding estimates for $t=s,s+1$ that the resolvent estimate for $t=s+2$ holds.
Finally, an application of the interpolation inequality leads to the resolvent estimates of the remaining cases.  

Now we prove the analyticity. Fixing $s \in \mathbb R$ and $t \in [s,s+2]$, for any $f \in H^s$ and $g \in H^{-t}$, there exist sequences $\{f_n\}, \{g_n\} \subset C_0^\infty(\mathbb R^2)$ such that $f_n \to f$ in $H^s$ and $g_n \to g$ in $H^{-t}$. Note that \begin{align*}
    \langle \chi \widetilde R_0( \lambda ) \chi f_n ,g_n\rangle=\frac{\rm i}{4}\int_{\mathbb R}\int_{\mathbb R}\chi(x)H_0^{(1)}(\lambda|x-y|)\chi(y)f_n(y)g_n(x)\,\mathrm{d}x\,\mathrm{d}y
\end{align*} 
and $\sqrt{s/2-{\rm i}\lambda|x-y|} \in \mathbb C \backslash (-\infty,0]$ for $\lambda \in \mathcal S$. Thus, the representation \eqref{Hankel} implies that $\langle \chi \widetilde R_0( \lambda ) \chi f_n ,g_n\rangle$ is analytic for $\lambda \in \mathcal S$. For any $K \subset \subset\mathcal S$, by the resolvent estimate derived above, we have $\langle \chi \widetilde R_0( \lambda ) \chi f_n ,g_n\rangle \to \langle \chi \widetilde R_0( \lambda ) \chi f ,g\rangle$ uniformly for $\lambda \in K$ as $n\to\infty$. Hence, $\langle \chi \widetilde R_0( \lambda ) \chi f ,g\rangle$ is also analytic for $\lambda \in \mathcal S$. Therefore, $\chi \widetilde R_0( \lambda ) \chi : H^s \to H^t$ is an analytic family of operators for $\lambda \in \mathcal S$.
The proof is complete.
\end{proof}

According to the decomposition of the resolvent of the biharmonic operator 
\[R_0(\lambda)(f)=\frac{1}{2\lambda^2} \Big(\widetilde R_0(\lambda)(f)-\widetilde R_0({\rm i}\lambda)(f)\Big),\] 
 the following lemma on analyticity and estimates for $R_0(\lambda)$ in two and three dimensions
is a direct consequence of Lemma \ref{lem1} and \ref{lemma2} for the Laplacian operator. The proof is omitted for brevity.
\begin{lemma}\label{lem3}
   \begin{enumerate}
       \item In three-dimensional case, the free resolvent $R_0(\lambda)$ is analytic for $\lambda \in \mathbb{C}\backslash \{0\}$ with $\Im \lambda>0, \Re\lambda>0$ as a family of operators \[
R_0(\lambda): H^s(\mathbb{R}^3) \to H^{t}(\mathbb{R}^3)
 \] for all $s \in \mathbb{R}$ and $ t \in [s,s+4]. $ Furthermore, given a cut-off function $\chi \in C_0^\infty(\mathbb{R}^3)$, $R_0(\lambda)$ can be extended to a family of analytic operators for $\lambda \in \mathbb{C}\setminus\{0\}$ as follows \[
 \chi R_0(\lambda) \chi : H^s(\mathbb{R}^3) \to H^t(\mathbb{R}^3)
 \] with the resolvent estimates 
 \begin{equation*}
     \|\chi R_0(\lambda) \chi \|_{\mathcal{L}(H^s,H^t)} \lesssim |\lambda|^{-2}(1+|\lambda|)^{t-s-1}(e^{L (\Im \lambda)_-} + e^{L (\Re \lambda)_-}),
 \end{equation*} where $s \in \mathbb R$ and $t\in [s, s+4]$, $v_-:=\max\{-v,0\}$ and $L$ is a constant larger than the diameter of the support of $\chi$.
 \item In two-dimensional case, the free resolvent $R_0(\lambda)$ is analytic for $\lambda \in \mathcal S$ with $\Im \lambda>0, \Re\lambda>0$ as a family of operators \[
R_0(\lambda): H^s(\mathbb{R}^2) \to H^{t}(\mathbb{R}^2)
 \] for all $s \in \mathbb{R}$ and $ t \in [s,s+4]. $ Furthermore, given a cut-off function $\chi \in C_0^\infty(\mathbb{R}^2)$, $R_0(\lambda)$ can be extended to a family of analytic operators for $\lambda \in \mathcal S$ as follows \[
 \chi R_0(\lambda) \chi : H^s(\mathbb{R}^2) \to H^t(\mathbb{R}^2)
 \] with the resolvent estimates 
 \begin{equation*}
     \|\chi R_0(\lambda) \chi \|_{\mathcal{L}(H^s,H^t)} \lesssim (\alpha(\lambda)^{-\frac{1}{2}}+\alpha({\rm i}\lambda)^{-\frac{1}{2}})|\lambda|^{-2}(1+|\lambda|)^{t-s}(e^{L (\Im \lambda)_-} + e^{L (\Re \lambda)_-}),
 \end{equation*} where $s \in \mathbb R$ and $t\in [s, s+4]$, $v_-:=\max\{-v,0\}$ and $L$ is a constant larger than the diameter of the support of $\chi$.\end{enumerate}
\end{lemma} 

In what follows, based on the above analysis of the free resolvent, we study the resolvent $R_V(\lambda)=(\Delta^2-\lambda^4+V)^{-1}$ for the biharmonic Schr\"odinger
operator. 

Since  $V \in L^\infty_{\rm comp}(\mathbb R^d)$ and $Vg \in L^2_{\rm comp}(\mathbb R^d)$ for $g\in H^s(\mathbb R^d)$ with $s>0$,
the following lemma follows from a direct application of the perturbation argument in \cite[Theorem 2.2 and 2.10]{DZ}. Thus, we omit the proof.
\begin{lemma}\label{resolvent}
Assume $ V \in L^\infty_{\rm comp}(\mathbb{R}^d)$, d = 2, 3. We have: 
      \begin{enumerate}
          \item In three-dimensional case, the resolvent $ R_{ V}(\lambda)$ is meromorphic for $\lambda \in \mathbb{C}\backslash \{0\}$ with $\Im \lambda>0, \Re\lambda>0$ as a family of operators $ R_{ V}(\lambda): H^s \to L^2$ for $s \in [-4,0]$. Moreover, it can be extended to a family of meromorphic operators as \[
       R_{ V}(\lambda): H^s_{\rm comp} \to L^2_{\rm loc},\quad \lambda \in \mathbb C \backslash \{0\},
      \] where $s \in [-4,0]$. Further, fixing a cut-off function $\chi \in C_0^\infty(\mathbb{R}^3)$, the resolvent $\chi  R_{ V}(\lambda) \chi: H^s\to L^2$ is analytic for $\lambda \in \mathscr{S}\cap -{\rm i}\mathscr{S}$,  where
     \[
\mathscr{S}:=\{\lambda \in \mathbb{C}\backslash \{0\}: \Im \lambda \ge -A-\delta \log(1+|\lambda|), |\lambda|>C_0\}.
 \]     Moreover,
      there exist constants $A,C_0,\delta,C$ and $L$ such that the resolvent estimates \begin{equation*}
     \|\chi  R_{ V}(\lambda) \chi \|_{H^s \to L^2} \le C |\lambda|^{-2}(1+|\lambda|)^{-s-1}e^{L (\Im \lambda)_-}
 \end{equation*} hold for $s \in [-4,0]$, $t \in [s,s+4]$ and $\lambda \in \mathscr{S}\cap -{\rm i}\mathscr{S}$. The constants $A,\delta,L$ depend on $\chi$ and $supp\, V$. The constants $C,C_0$ depend on $\chi$ and $\| V\|_{L^\infty}$.
 \item In two-dimensional case, the resolvent $ R_{ V}(\lambda)$ is meromorphic for $\lambda \in \mathcal S$ with $\Im \lambda>0, \Re\lambda>0$ as a family of operators $ R_{ V}(\lambda): H^s \to L^2$ with $s \in [-4,0]$ . It can be further extended to a meromorphic family of operators as \[
       R_{ V}(\lambda): H^s_{\rm comp} \to L^2_{\rm loc},\quad \lambda \in \mathbb C \backslash \{0\},
      \] where $s \in [-4,0]$ . Fix a cut-off function $\chi \in C_0^\infty(\mathbb{R}^2)$. The resolvent $\chi  R_{ V}(\lambda) \chi: H^s\to L^2$ is analytic for 
      $\lambda \in \mathscr{S}\cap-{\rm i}\mathscr{S}$. Moreover,
      there exist constants $A,C_0,\delta,C$ and $L$ such that the resolvent estimates \begin{equation*}
     \|\chi R_{V}(\lambda) \chi \|_{H^s \to L^2} \le C |\lambda|^{-\frac{5}{2}}(1+|\lambda|)^{-s}e^{L (\Im \lambda)_-} 
 \end{equation*} hold for $s \in [-4,0]$   and $\lambda \in \mathscr{S}\cap-{\rm i}\mathscr{S}$. The constants $A,\delta,L$ depend on $\chi$ and $supp\, V$. The constants $C,C_0$ depend on $\chi$, $\| V\|_{L^\infty}$.\end{enumerate}
 \end{lemma} 
 

The following theorem concerns the well-posedness of the direct scattering problem \eqref{DP}. Moreover, it provides a Born series expansion of the solution, which is useful in the subsequent study of the inverse problem.
 
\begin{theorem}\label{wp}
    Assume that the random source $f$ is a GMIG random field of order $-m$ with $m>d-8$ in $D \subset \mathbb{R}^d$ and $V \in L^\infty_{\rm comp}(\mathbb{R}^d)$. For $k$ sufficiently large, the scattering problem \eqref{DP}--\eqref{src} almost surely admits a unique solution $u\in H^{s+4}_{\rm loc}(\mathbb{R}^d)$ with constant $s$ satisfying $s<(m-d)/2$. Furthermore, the solution $u$ can be almost surely expanded as the Born series \begin{align}\label{born}
        u=\sum_{j \ge 0}(-R_0(k) V)^j R_0(k) f.
    \end{align}
\end{theorem}
\begin{proof}
  There holds the resolvent identity \begin{equation}\label{RI2}
\chi R_{V}(k)\chi =  (I  + \chi R_0(k) V)^{-1}\chi R_0(k) \chi,
\end{equation} where $\chi \in C_0^\infty(\mathbb R^d)$ is a cut-off function
such that $\chi =1$ in ${\rm supp}\, f \cup {\rm supp}\, V$. Denoting \[
 u:= R_{ V}(k) (f),
\]  from \eqref{RI2} we have that \begin{align*}
    \chi u=(I  + \chi R_0(k) V)^{-1}\chi R_0(k) \chi f.
\end{align*} For sufficiently large $k$, we have \begin{align}\label{Ns}
    \chi u &=(I  + \chi R_0(k) V)^{-1}\chi R_0(k) \chi f=\sum_{j=0}^\infty(-\chi R_0(k)V)^j(\chi R_0(k) \chi f) \notag\\ &=\sum_{j=0}^\infty (- R_0(k)V)^j( R_0(k)  f).
\end{align} By Lemma \ref{lem3}, the series $\sum_{j=0}^\infty(-\chi R_0(k)V)^j $ is convergence in $\mathcal L(L^2,L^2)$ if $k$ is large enough. On the other hand, by Lemma \ref{Lem2.1}, we know that $f \in H^s$ with non-positive $s$ satisfying $s <(m-d)/2$. Then we get $\chi R_0(k) \chi f \in L^2$ since $m>d-8$. Thus, $ \chi u$ given by \eqref{Ns} is well-defined in $L^2$, which implies $u$ given by \eqref{born} is well-defined in $L^2_{\rm loc}$. Furthermore, $u \in H^{s+4}_{\rm loc}$ by elliptic regularity theory. It can also be verified $u$ is a solution to the direct scattering problem \eqref{DP}--\eqref{src}. 
 Moreover, the uniqueness of the direct scattering problem can be obtained from \cite{LW_2}.  The proof is complete.
\end{proof}

\section{Stability for the inverse random source problem}\label{ip}

In this section, we investigate the stability for the inverse random source problem. To achieve this goal, we consider random sources which satisfy the following
assumption.

\emph{Assumption (A). The random source $f$ is a real-valued microlocally isotropic Gaussian random function of order $-m$ in $D$
 such that $m>d-1$. Its covariance operator has the symbol $c_f(x,\xi) = h(x)|\xi|^{-m} + a_f(x, \xi)$ with principal symbol $h(x)|\xi|^{-m}$ which satisfies the conditions:  (i) $|h(x)| \le M$ for $x \in D$; (ii) $|c_f(x, \xi)| \le M(1+|\xi|)^{-m}$ for $\xi \in \mathbb{R}^d$ and $x \in D$; (iii) $|a_f(x, \xi)| \le M |\xi|^{-(m+1)}$ for $|\xi| \ge 1$ and $x \in D$. }

As shown in \cite{RSWXZ}, the kernel $K_f(x,y)$ can be represented as a sum of a singular part and a bounded continuous remainder. 
\begin{lemma}\label{Kernel}
	Let the random source $f$ satisfy \emph{\emph{Assumption (A)}}. The covariance function $K_f(x,y)$ has the following form:
	
	(i)  If $a<m-d+1\le a+1$ with $a=1,2,3,...$,  \[
	K_f(x, y)= ch(x)|x-y|^{m-d} +F_m(x,y),
	\] where $c$ is a constant dependent on $m$ and $F_m(x,y) \in C^{a,\gamma}(\mathbb{R}^d \times \mathbb{R}^d)$ with $\gamma \in (0, m-d+1-a)$.
	
	(ii)
	If $d-1<m< d$,  \[
	K_f(x, y)= ch(x)|x-y|^{m-d} +F_m(x,y),
	\] where $c$ is a constant dependent on $m$ and $F_m(x,y) \in C^{0,\gamma}(\mathbb{R}^d \times \mathbb{R}^d)$ with $\gamma \in (0, m-d+1)$. 
	
	(iii)
	If $m=d$,  \[
	K_f(x, y)= ch(x)\log{|x-y|} +F_m(x,y),
	\] where $c$ is a constant dependent on $m$ and $F_m(x,y) \in C^{0,\gamma}(\mathbb{R}^d \times \mathbb{R}^d)$ with $\gamma \in (0, 1)$. 
	
	For all of the above three cases, we have \[
	\|F_m\|_{L^\infty(D \times D)} \lesssim M. 
	\]
\end{lemma}

Now we formulate the inverse source problem. The radiating condition in \eqref{src} yields that the solution $u$ has the following asymptotic expansion as $|x|\to\infty$:
\[
u(x) = \frac{e^{{\rm i}k|x|}}{|x|^{\frac{d-1}{2}}} u^\infty(\hat{x},k) + O\Big(\frac{1}{|x|^\frac{d+1}{2}}\Big),
\]
where $u^\infty(\hat{x},k)$ is referred as the far-field pattern with observation direction $\hat{x}= \frac{x}{|x|}$. Let $I$ be a bounded interval of frequencies. 
We consider the following inverse random source problem by multi-frequency far-field data. 

{\bf \it Inverse Random Source Problem:} 
Determine the micro-correlation strength $h$ from the multi-frequency far-field pattern $\{u^\infty(\hat{x},k): \hat{x}\in\mathbb S^{d-1}, \,\, k\in I\}$. 

Our analysis starts with the Born series expansion obtained in Theorem \ref{wp}\[
  u=\sum_{j \ge 0}(-R_0(k) V)^j R_0(k) f,
\] which gives the Born series representation of the far-field pattern 
\begin{align}\label{born_1}
u^\infty(\hat{x},k)&= -\frac{C_d}{2 k^{\frac{7-d}{2}}}  \int_{\mathbb{R}^d} e^{-{\rm i}k\hat{x} \cdot y}\sum_{j \in \mathbb{N}}[(-V R_0(k) )^{j} f](y)\,\mathrm{d}y \notag\\ &= -\frac{C_d}{2 k^{\frac{7-d}{2}}}\int_{\mathbb{R}^d} e^{-{\rm i}k\hat{x} \cdot y} f(y)\,\mathrm{d}y-\frac{C_d}{2 k^{\frac{7-d}{2}}}  \int_{\mathbb{R}^d} e^{-{\rm i}k\hat{x} \cdot y}[-VR_0(k)f](y)\,\mathrm{d}y \notag\\ 
& \quad-\frac{C_d}{2 k^{\frac{7-d}{2}}}\int_{\mathbb{R}^d} e^{-{\rm i}k\hat{x} \cdot y}\sum^\infty_{j=2}[(-V R_0(k) )^{j} f](y)\,\mathrm{d}y \notag\\ &:=u_0^\infty(\hat{x},k) + u_1^\infty(\hat{x},k) + u_2^\infty(\hat{x},k),
\end{align} where the constant $C_d$ is given by 
\begin{align*}
    C_d=\begin{cases}
        \frac{e^{{\rm i}\frac{\pi}{4}}}{\sqrt{8\pi}}, &\quad d=2,\\[10pt] 
        \frac{1}{4\pi}, &\quad d=3.
    \end{cases}
\end{align*}

We first consider the zeroth-order term $u_0^\infty(\hat{x},k)$ in the Born series expansion \eqref{born_1} of the far-field pattern.
For a fixed positive constant $\tau>0$, by taking expectation of the correlation of the zeroth-order term we obtain
\begin{align*}
	&\mathbb E[u_0^\infty(\hat{x}, k+\tau) \overline{u_0^\infty(\hat{x}, k)}]\\
	&=\frac{C_d^2}{4k^{\frac{7-d}{2}}(k+\tau)^{\frac{7-d}{2}}}\int_{\mathbb R^d}\int_{\mathbb R^d} e^{-{\rm i}(k+\tau)\hat{x}\cdot y}  e^{{\rm i}k\hat{x}\cdot z}
	\mathbb E[f(y)f(z)]{\rm d}y{\rm d}z\\
	&= \frac{C_d^2}{4k^{\frac{7-d}{2}}(k+\tau)^{\frac{7-d}{2}}}\int_{\mathbb R^d} \Big[\int_{\mathbb R^d} K_f(y, z)e^{-{\rm i}k\hat{x}\cdot(y-z)}{\rm d}z\Big]
	e^{-{\rm i}\tau\hat{x}\cdot y}{\rm d}y \\
	&= \frac{C_d^2}{4k^{\frac{7-d}{2}}(k+\tau)^{\frac{7-d}{2}}} \Big[\int_{\mathbb R^d} h(y) e^{-{\rm i}\tau\hat{x}\cdot y}{\rm d}y |k\hat{x}|^{-m} + \int_{\mathbb R^d} a_f(y, k\hat{x}) 
	e^{-{\rm i}\tau\hat{x}\cdot y}{\rm d}y \Big]\\
	&= \frac{C_d^2}{4k^{\frac{7-d}{2}}(k+\tau)^{\frac{7-d}{2}}k^m} \widehat{h}(\tau\hat{x}) + O\left(\frac{1}{k^{m+8-d}}\right),\end{align*} which gives \begin{align}\label{zeroorder}
	    \widehat{h}(\tau\hat{x})=\frac{4k^{\frac{7-d}{2}}(k+\tau)^{\frac{7-d}{2}}k^m}{C_d^2}\mathbb E[u_0^\infty(\hat{x}, k+\tau) \overline{u_0^\infty(\hat{x}, k)}]+O(k^{-1}).
	\end{align}
	It should be noted that \eqref{zeroorder} establishes a relation between the Fourier transform of the strength $h$ and the zeroth-order term of the far-field pettern.
	
Now we consider the difference of the expectation of the correlation of the far-field pattern and its corresponding zeroth-order part. Specifically,
 denoting $C(k,\tau,d)=\frac{4k^{\frac{7-d}{2}}(k+\tau)^{\frac{7-d}{2}}k^m}{C_d^2}$,
 from the Cauchy-Schwarz inequality  we have 
 \begin{align}
	&C(k,\tau,d)\Big|\mathbb E[u^\infty(\hat{x}, k+\tau) \overline{u^\infty(\hat{x}, k)}] -\mathbb E[u_0^\infty(\hat{x}, k+\tau) \overline{u_0^\infty(\hat{x}, k)}]\Big| \notag\\ &\lesssim  \sum_{j=1}^2\Big(\sqrt{C(k,\tau,d)\mathbb E |u_j^\infty(\hat{x},k)|^2}\sqrt{C(k,\tau,d)\mathbb E |u_0^\infty(\hat{x},k+\tau)|^2} \notag\\ &\quad+\sqrt{C(k,\tau,d)\mathbb E |u_j^\infty(\hat{x},k+\tau)|^2}\sqrt{C(k,\tau,d)\mathbb E |u_0^\infty(\hat{x},k)|^2}\notag \\ &\quad+C(k,\tau,d)\mathbb E |u_j^\infty(\hat{x},k)|^2+C(k,\tau,d)\mathbb E |u_j^\infty(\hat{x},k+\tau)|^2\Big)\notag \\ &\lesssim \sum_{j=1}^2\Big(\sqrt{C(k,\tau,d)\mathbb E |u_j^\infty(\hat{x},k)|^2}+\sqrt{C(k,\tau,d)\mathbb E |u_j^\infty(\hat{x},k+\tau)|^2}\notag \\ &\quad+C(k,\tau,d)\mathbb E |u_j^\infty(\hat{x},k)|^2+C(k,\tau,d)\mathbb E |u_j^\infty(\hat{x},k+\tau)|^2\Big), \label{CS}
\end{align} where we have used the fact that for large $k$, 
\[
C(k,\tau,d)\mathbb E |u_0^\infty(\hat{x},k)|^2 \lesssim \widehat h(0) \lesssim 1.
\] 

To estimate $\mathbb E |u_j^\infty(\hat{x},k)|^2$ with $j=1, 2$ in the difference \eqref{CS},
we apply resolvent estimate in Lemma \ref{lem3} and obtain the following lemma.
\begin{lemma} \label{2dim12} The following estimates hold for large $k$:
	\[ \mathbb E|u^\infty_1(\hat{x},k)|^2 \lesssim \frac{1}{k^{10}},\quad \mathbb E|u^\infty_2(\hat{x},k)|^2 \lesssim \frac{1}{k^{13+d}}. \]
\end{lemma} \begin{proof} We only prove the second inequality and the first one follows in a similar way. Denote the Schwarz kernel of the operator $\sum_{j \ge 2} (V R_0(k))^j$ by $G_{2,k}(x,y)$. Then we have \begin{align}\label{4.5}
		\mathbb E |u^\infty_2(\hat{x},k)|^2=\frac{C_d^2}{4k^{7-d}}\int_{\mathbb{R}^{4d}}e^{-{\rm i}k \hat{x} \cdot (y-z)} G_{2,k}(y,s) \overline{G_{2,k}(z,t)} K_f(s,t)\,\mathrm{d}s\,\mathrm{d}t\,\mathrm{d}y\,\mathrm{d}z.
	\end{align} By the resolvent estimate in Lemma \ref{lem3} we have  \[
	\|V R_0(k) \phi\|_{L^2 } \lesssim \frac{1}{k^\frac{d+3}{2}}, \quad \forall \phi \in L^2_{\rm comp},
	\] which gives \begin{equation}\label{4.6}
		\Big\|\sum_{j \ge 2} (V R_0(k))^j \phi \Big\|_{L^2} =\Big\|\int_{\mathbb{R}^d} G_2(x,y) \phi(y)\,\mathrm{d}y \Big\|_{L^2} \lesssim \frac{1}{k^{d+3}}\|\phi \|_{L^2}, \quad \forall \phi \in L^2_{\rm comp}.
	\end{equation} By inserting \eqref{4.6} into \eqref{4.5} we obtain \begin{align*}
		\mathbb E |u^\infty_2(\hat{x},k)|^2 &\lesssim \frac{1}{k^{7-d}}\Big\| \int_{\mathbb{R}^d}\int_{\mathbb{R}^d}G_2(y,s)\overline{G_2(z,t)}K_f(s,t)\,\mathrm{d}s\,\mathrm{d}t\Big\|_{L^2(\mathbb R^d \times \mathbb R^d)} \\ & \lesssim \frac{1}{k^{10}} \Big\| \int_{\mathbb{R}^d}\overline{G_2(z,t)} K_f(s,t)\,\mathrm{d}t \Big\|_{L^2(\mathbb R^d \times \mathbb R^d)} \\  &\lesssim \frac{1}{k^{13+d}} \|K_f(s,t)\|_{L^2(\mathbb R^d \times \mathbb R^d)} \lesssim \frac{1}{k^{13+d}},
	\end{align*} where the assumption $m>d-1$ and Lemma \ref{Kernel} are used in the last step. The proof is complete.
	\end{proof}
 
Combining \eqref{zeroorder} and Lemma \ref{2dim12}, we derive the following estimate for $d-1<m<d+3$ \begin{align}\label{tail}
      \widehat{h}(\tau\hat{x})=\frac{4k^{\frac{7-d}{2}}(k+\tau)^{\frac{7-d}{2}}k^m}{C_d^2}\mathbb E[u^\infty(\hat{x}, k+\tau) \overline{u^\infty(\hat{x}, k)}]+O(k^{-\delta_1})
 \end{align} with $\delta_1<\min\{1,(d+3-m)/2\}$ being a positive constant. Then letting $\tau=\eta k$ with $\eta \in (0,1)$ in \eqref{tail} we obtain the following estimate
 \begin{align}\label{STEP1}
	|\hat{h}(\xi)| \lesssim  \sup_{\eta \in (0,1), \hat{x} \in \mathbb{S}^{d-1}} \Big|k^{m+7-d}\mathbb E[u^\infty(\hat{x}, (1+\eta)k) \overline{u^\infty(\hat{x}, k)}]\Big| + \frac{1}{k^{\delta_1}},
\end{align} which holds for all $|\xi| <k$.

Based on the estimate \eqref{STEP1}, we introduce
the data discrepancy in a finite interval $I = (K_0, K]$ of frequencies with $C_0<K_0<K$ as follows
\[
\varepsilon_1^2 =  \sup_{k\in I, \eta\in(0, 1), \hat{x}\in\mathbb S^{d-1}} \varepsilon_1^2(k, \eta, \hat{x}),
\]
where
\[
\varepsilon_1^2(k, \eta, \hat{x}) := \Big|k^{m+7-d}\mathbb E[u^\infty(\hat{x}, (1+\eta)k) \overline{u^\infty(\hat{x}, k)}]\Big|^2.
\] Here $C_0$ is a sufficiently large constant specified in Lemma \ref{resolvent}.
Since both the source function and the potential function are real-valued, we have $\overline{u(x, k)} = u(x,{\rm i}k)$ for $k>0$. Note that the far-field pattern can be represented as \[
u^\infty(\hat x,k)=\frac{C_d}{2k^{\frac{7-d}{2}}} \int_{\mathbb R^d} e^{-{\rm i}k\hat x\cdot y}(-V(y)u(y,k)+f(y))\,\mathrm{d}y,
\] which gives \[
\overline{u^\infty(\hat x,k)}=\frac{\overline{C_d}}{2k^{\frac{7-d}{2}}} \int_{\mathbb R^d} e^{{\rm i}k\hat x\cdot y}(-V(y)u(y,{\rm i} k)+f(y))\,\mathrm{d}y
\] when $k>0$. Hence, denoting 
\begin{align}\label{v}
v^\infty(\hat x,k):=\frac{\overline{C_d}}{2k^{\frac{7-d}{2}}} \int_{\mathbb R^d} e^{{\rm i}k\hat x\cdot y}(-V(y)u(y,{\rm i}k)+f(y))\,\mathrm{d}y,
\end{align}
we can meromorphically extend $\varepsilon_1^2(\cdot,\eta, \hat{x})$ from $\mathbb R^+$ to $\mathbb C$ as follows
\begin{equation}\label{extension}
\varepsilon_1^2(k, \eta, \hat{x}) = k^{2(m+7-d)}  \mathbb E[u^\infty(\hat{x}, (1+\eta)k) v^\infty(\hat{x}, k)]
\mathbb E[v^\infty(\hat{x}, (1+\eta)k) u^\infty(\hat{x}, k)], \quad k\in\mathbb C.
\end{equation}
Notice that for $k>0$ the extended data $\varepsilon_1^2(k, \eta, \hat{x})$  in \eqref{extension} becomes the originally introduced data discrepancy as follows
\[
\varepsilon_1^2(k, \eta, \hat{x})=  \Big|k^{m+7-d}\mathbb E[u^\infty(\hat{x}, (1+\eta)k) \overline{u^\infty(\hat{x}, k)}]\Big|^2.
\] 

Denote an infinite slab in $\mathbb C$ as follows
\begin{equation}\label{regionR}
	\mathcal{R} = \{z\in \mathbb C: (K_0, +\infty)\times (-h_0, h_0) \},
\end{equation}
for some constants $K_0>0$ and $h_0>0$.
The following lemma in \cite{LZZ} provides a quantitative analytic continuation principle in an infinite slab.

\begin{lemma}\label{QAC}
	Let $p(z)$ be analytic in the infinite rectangular slab $\mathcal{R}$
	and continuous in $\overline{\mathcal{R}}$ satisfying
	\begin{align*}
		\begin{cases}
			|p(z)|\leq \epsilon, &\quad z\in (K_0, K],\\
			|p(z)|\leq M, &\quad z\in \mathcal{R},
		\end{cases}
	\end{align*}
	where $K_0, K, \epsilon$ and $M$ are positive constants. Then there exists a function $\mu(z)$ with $z\in (K, +\infty)$ satisfying 
	\begin{equation*}
		\mu(z) \geq \frac{64ah}{3\pi^2(a^2 + 4h^2)} e^{\frac{\pi}{2h}(\frac{a}{2} - z)},
	\end{equation*}
	where $a = K - K_0$ such that
	\begin{align*}
		|p(z)|\leq M\epsilon^{\mu(z)},\quad \forall\, z\in (K, +\infty).
	\end{align*}
\end{lemma} 

In the following lemma, we show that $\varepsilon_1^2(k, \eta, \hat{x})$ is analytic and has an upper bound for $k \in \mathcal{R}$.
\begin{lemma}\label{bded}
	There exists an infinite slab $\mathcal{R}$ defined by \eqref{regionR} such that $\mathbb{E} [u^\infty(\hat{x}, (1+\eta)k) v^\infty(\hat{x}, k)]$ is analytic for $k \in \mathcal{R}$ and $\eta \in (0,1)$. Furthermore, there holds the inequality \[
	|\varepsilon_1^2(k, \eta, \hat{x})| \lesssim |k|^{2m},\quad k \in \mathcal{R},\quad \eta \in 
(0,1).
	\] 
\end{lemma}
\begin{proof}
	Notice that \begin{align*}
		\mathbb{E} [u^\infty(\hat{x}, (1+\eta)k) v^\infty(\hat{x}, k)]  =\frac{|C_d|^2}{4 k^{7-d}}\int_{\mathbb{R}^d}\int_{\mathbb{R}^d}e^{{\rm i}k\hat{x}\cdot( z-((1+\eta)y)}J(y,z;k)\,\mathrm{d}y\,\mathrm{d}z,
	\end{align*} where \begin{align*}
		&J(y,z;k)\\ &:=K_f(y,z)-V(z)\mathbb E [f(y)u(z,{\rm i}k)] -V(y)\mathbb E [f(z)u(y,(1+\eta)k)]+\\
		&\quad +V(y)V(z) \mathbb E[u(y,(1+\eta)k)u(z,{\rm i}k)] \\&=K_f(y,z)-\int_{\mathbb{R}^d} V(z) G_{V,{\rm i}k}(z,s)K_f(y,s)\,\mathrm{d}s-\int_{\mathbb{R}^d} V(y) G_{V,(1+\eta)k}(t,y)K_f(t,z)\,\mathrm{d}t \\ &\quad+\int_{\mathbb{R}^d}\int_{\mathbb{R}^d}V(y)V(z) K_f(s,t)G_{V,{\rm i}k}(z,s)G_{V,(1+\eta)k}(y,t)\,\mathrm{d}s\,\mathrm{d}t.
	\end{align*} Here $G_{V,k}(x,y)$ is the Schwarz kernel of the resolvent $R_V(k)$. Note that there exist $h_0$ and $K_0$ such that 
	$\mathcal{R}\cup(1+\eta)\mathcal{R}\subset \mathscr{S} \cap -{\rm i}\mathscr{S}$ and ${\rm i}\mathcal{R}\cup{\rm i}(1+\eta)\mathcal{R}\subset \mathscr{S}\cap -{\rm i}\mathscr{S}$, where $\mathcal R$ is the
	infinite slab and $\mathscr S$ is the logarithmic domain in Lemma \ref{resolvent}.
	Moreover, using Lemma \ref{resolvent} we have that  \[ 
	[V R_V(k)(\phi)](x)=\int_{\mathbb{R}^d}V(x) G_{V,k}(x,y) \phi(y)\,\mathrm{d}y,\quad \forall \, \phi \in L^2_{\rm comp}
	\] is analytic for $k \in \mathscr{S}$ as a family of functions in $L^2$ with the following estimate \[
	\|V R_V(k)(\phi)\|_{L^2}=\Big\|\int_{\mathbb{R}^d}V(x) G_{V,k}(x,y) \phi(y)\,\mathrm{d}y\Big\|_{L^2} \lesssim \frac{1}{|k|^\frac{d+3}{2}}\|\phi\|_{L^2}.
	\] Hence, there exists an infinite slab $\mathcal{R}$ such that $J(y,z;k)$ is analytic for $k \in \mathcal R$ as a family of functions in $L^2(\mathbb{R}^d \times \mathbb{R}^d)$ with \begin{align*}
		\|J\|_{L^2(\mathbb{R}^d \times \mathbb{R}^d)} \lesssim \|K_f\|_{L^2(\mathbb R^d \times \mathbb R^d)}\left(1+\frac{1}{k^\frac{d+3}{2}}+\frac{1}{k^{d+3}}\right) \lesssim 1, \quad k \in \mathcal{R},
	\end{align*} where we have used Lemma \ref{Kernel}. Thus, we have that $\mathbb{E} [u^\infty(\hat{x}, (1+\eta)k) v^\infty(\hat{x}, k)]$ is analytic for $k \in \mathcal{R}$ with \begin{align*}
		&|\mathbb{E} [u^\infty(\hat{x}, (1+\eta)k) v^\infty(\hat{x}, k)] | \\
		&\lesssim |k|^{d-7}\|e^{{\rm i}k \hat{x} \cdot (z-((1+\eta)y)}\|_{L^2(D \times D)}\|J\|_{L^2(D \times D)} \lesssim |k|^{d-7}.
	\end{align*} Similarly, we can obtain the analyticity and estimate for $\mathbb{E} [v^\infty(\hat{x}, (1+\eta)k) u^\infty(\hat{x}, k)]$, which completes the proof
	noticing \eqref{extension}.
\end{proof} 

Introduce the function space $\mathcal{C}_s:=\{h \in C_0^\infty(\mathbb R^d): \|h\|_{H^s} \le M\}$ where $M>0$ and $s>0$ are two positive constants.
In the following theorem, we present a quantitative stability estimate for the inverse random source problem.
\begin{theorem}\label{Thm3.1}
	Assume that $V \in L^\infty_{\rm comp}(\mathbb R^d)$ and the random source $f$ satisfies \emph{\emph{Assumption (A)}} with $d-1<m<d+3$ and $h \in \mathcal{C}_s$. 
	The following increasing stability estimate holds:
	\begin{align}\label{stab0}
		\|h\|_{L^2(\mathbb{R}^3)}^2
		\lesssim K^{\beta_0}\varepsilon_1^2+\frac{1}{K^\beta(\ln|\ln \varepsilon_1|)^\beta}
	\end{align}
	with $\beta_0=dt+2m$ and $\beta = \min\{\frac{\delta_1}{2}-\frac{dt}{2}, st\}$. Here the positive constants $\delta_1,t$ satisfy $\delta_1<\min\{1,(d+3-m)/2\}$ and $0<t<\frac{\delta_1}{d}$.
\end{theorem}
\begin{proof}
	By applying the quantitative analytic continuation principle in Lemma \ref{QAC} to $\varepsilon_1^2(z, \eta, \hat{x})$, we obtain
	\[
	|\varepsilon_1^2(k, \eta, \hat{x})|\lesssim k^{2m}\varepsilon_1^{2\mu({k})},\quad {k}>K,
	\]
	where 
	\[
	\mu({k})=\frac{64ah}{3\pi^2(a^2+4h^2)}e^{\frac{\pi}{2 h}(\frac{a}{2}-{k})}, \quad a = K - K_0.
	\]
	Fixing a number $A$, since for ${k}\in (K,A]$ it holds
	\[
	\mu({k})\gtrsim ce^{-\sigma A}
	\]
	for some constant $\sigma>0$, we have
	\[
	|\varepsilon_1^2(k, \eta, \hat{x})|\lesssim k^{2m} \exp\{-ce^{-\sigma A}|\ln \varepsilon_1|\}.
	\]
	Using the fact 
	$e^{-x}\leq\frac{6!}{x^{6}}$ for $x>0$, we obtain the following estimate
	\[
	|\varepsilon_1^2(k, \eta, \hat{x})|\lesssim k^{2m} e^{6\sigma A}|\ln \varepsilon_1|^{-6},\quad {k}\in (K,A],
	\]
	which holds uniformly for $\eta\in (0, 1)$ and $\hat{x}\in\mathbb S^{d-1}$. Now we derive the increasing stability estimate using the analytic continuation 
	developed in \cite{ZZ}. 
	Specifically, we consider the following two situations.
	
	\textbf{Case 1:} $K\leq \frac{1}{2\sigma}\ln|\ln \varepsilon_1|$.
	Taking $A=\frac{1}{2\sigma}\ln|\ln \varepsilon_1|$, we have
	\[
	|\varepsilon_1^2(A,  \eta, \hat{x})|\lesssim A^{2m}|\ln\varepsilon_1|^{-3}.
	\]
	Then from \eqref{STEP1} we have the following estimate
	\[
	| \widehat{h}(  \xi)|^2\leq A^{2(m+7-d)}|\ln\varepsilon_1|^{-3}+\frac{1}{A^{\delta_1}},
	\]
	which holds for all $  \xi$ satisfying $| \xi|\leq A$.
	Let $t<\frac{\delta_1}{d}$ be a positive constant.
	Then we have
	\[
	\int_{|  \xi|\leq A^{t}}| \widehat{h}(  \xi)|^2\mathrm{d}  \xi\lesssim (\ln|\ln\varepsilon_1|)^{dt+2m}|\ln\varepsilon_1|^{-3}+(\ln|\ln\varepsilon_1|)^{-(\delta_1-3t)}\lesssim \frac{1}{(\ln|\ln\varepsilon_1|)^{2\beta_1}}.
	\]
	Here we denote $\beta_1=\frac{\delta_1}{2}  - \frac{dt}{2}$.
	Since $h \in \mathcal{C}_s$, we have
	\[
	\int_{| \xi|> A^{t}}| \widehat{h}( \xi)|^2\mathrm{d} \xi\lesssim\frac{1}{A^{2st}}\leq (\ln|\ln \varepsilon_1|)^{-2\beta_2},
	\]
	where $\beta_2 = st$. Denote $\beta = \min\{\beta_1, \beta_2\}$.
	Combining the above estimates, we obtain
	\[
	\begin{split}
		\|h\|_{L^2}^2=&\int_{| \xi|\leq A^{t}}| \widehat{h}( \xi)|^2\mathrm{d} \xi+\int_{| \xi|> A^{t}}| \widehat{h}( \xi)|^2\mathrm{d} \xi\\
		\lesssim &(\ln|\ln \varepsilon_1|)^{-2\beta}\\
		\lesssim &\frac{1}{K^\beta(\ln|\ln \varepsilon_1|)^\beta}.
	\end{split}
	\]
	\textbf{Case 2:} $K\geq \frac{1}{2\sigma}\ln|\ln \epsilon|$. First notice that
	\[
	\int_{| \xi|\leq K^{t}}| \widehat{h}( \xi)|^2\mathrm{d} \xi\lesssim K^{dt+2m}\varepsilon_1^2+K^{-2\beta_1}.
	\]
	Also, we have
	\[
	\int_{| \xi|> K^{t}}| \widehat{h}( \xi)|^2\mathrm{d} \xi\lesssim \frac{1}{K^{2st}}\leq K^{-2\beta_2}.
	\]
	Consequently, we arrive at
	\[
	\begin{split}
		\|h\|_{L^2}^2=&\int_{| \xi|\leq K^{t}}| \widehat{a}( \xi)|^2\mathrm{d} \xi+\int_{| \xi|> K^{t}}| \widehat{a}( \xi)|^2\mathrm{d} \xi\\
		\lesssim &K^{dt+2m}\varepsilon_1^2+K^{-2\beta}\\
		\lesssim &K^{dt+2m}\varepsilon_1^2+\frac{1}{K^\beta(\ln|\ln \varepsilon_1|)^\beta}.
	\end{split}
	\]
	Combining Case 1 and Case 2 we obtain the following stability estimate
	\begin{align*}
		\|h\|_{L^2}^2
		\lesssim K^{dt+2m}\varepsilon_1^2+\frac{1}{K^\beta(\ln|\ln \varepsilon_1|)^\beta}.
	\end{align*}
	The proof is complete.
\end{proof}
\begin{remark} 
	The stability estimate \eqref{stab0}  consists of the Lipschitz data discrepancy and a logarithmic type part. The latter illustrates the ill-posedness of the inverse random source
	problem. Observe that the logarithmic stability decreases which leads to the improvement of the stability estimate as the upper bound $K$ of the frequency increases,
	Moreover, the stability result \eqref{stab0} implies the uniqueness of the inverse problem.
\end{remark}

In what follows, we derive another stability estimate for the inverse random source problem by expectation of the far-field pattern averaged over the high frequency
band in three dimensions, which quantifies a uniqueness result implied in \cite{li2022far}. In this case, the range of $m$ in the principle symbol of the covariance operator of the random source
can be enlarged. In fact, in Lemma \ref{2dim12}, to handle the higher-order terms in Born series expansion of the far-field pattern, we employ resolvent estimates of the biharmonic 
Schr\"odinger operator. As a consequence, we obtain decay estimates for the expectations of the higher-order terms at high wavenumber, upon which the range
of $m$ in the principle symbol of the covariance operator of the random source is determined. However, 
by investigating the first-order term of the far-field pattern carefully without applying the resolvent estimates directly,
we can obtain the stability estimate by using another set of far-field data which enlarges the range of $m$.

Recall that the first-order term in the far-field pattern is \[
u^\infty_1(\hat x,k)=\frac{C_d}{2k^2}\int_{\mathbb R^3}\int_{\mathbb R^3} e^{-{\rm i}k \hat x \cdot y} V(y) G(y,s) f(s)\,\mathrm{d}s\,\mathrm{d}y.
\] 
Instead of applying resolvent estimates to $\mathbb E|u^\infty_1(\hat x,k)|^2$ directly, we will show that 
\begin{equation}\label{oneorder3}
	\lim_{k \to \infty}\frac{1}{k^{1-\delta_2}}\int_k^{2k}\int_{\mathbb S^2}t^{m+4}\mathbb E|u^\infty_1(\hat x,t)|^2\,\mathrm{d}\sigma(\hat x)\,\mathrm{d}t=0.
\end{equation} By the dominated convergence theorem, \eqref{oneorder3} holds once we have \begin{equation}\label{oneorder3e}\int_1^\infty\int_{\mathbb S^2} t^{m+3+\delta_2}\mathbb E|u^\infty_1(\hat x,t)|^2\,\mathrm{d}\sigma(\hat x)\,\mathrm{d}t<\infty.
\end{equation}   
To prove \eqref{oneorder3e},
consider the mollification $f_\epsilon := f * \chi_\epsilon$, where $\chi_\epsilon(x)=\frac{1}{\epsilon^3}\chi(\frac{x}{\epsilon})$ with $\chi \in C_0^\infty(\mathbb{R}^3)$ such that $\int_{\mathbb{R}^3} \chi \,\mathrm{d}x=1$. By the Fatou lemma, \eqref{oneorder3e} holds if \begin{equation}\label{oneorderE}
	\limsup_{\epsilon \to 0} \int_{\mathbb{S}^2}\int_1^\infty t^{m+3+\delta_2}\mathbb E|u^\infty_{1, \epsilon}(\hat x,t)|^2\,\mathrm{d}t\,\mathrm{d}\sigma(\hat x) < \infty,
\end{equation}  where \[
u^\infty_{1,\epsilon}(\hat x,k)=\frac{C_d}{2k^2}\int_{\mathbb R^3}\int_{\mathbb R^3} e^{-{\rm i}k \hat x \cdot y} V(y) G(y,s) f_\epsilon(s)\,\mathrm{d}s\,\mathrm{d}y.
\]  Taking $z=y-s$ and then letting $z=\rho \omega$, we obtain \begin{align*}
	u^\infty_{1, \epsilon}(\hat x,k) &= \frac{C_d}{2k^2}\int_{\mathbb{R}^3}\int_{\mathbb{R}^3}e^{-{\rm i}k\hat x \cdot y}V\left(y\right)f_\epsilon\left(y-z\right)\left(\frac{e^
		{{\rm i}k|z|}}{8\pi k^2|z|}-\frac{e^
		{-k|z|}}{8\pi k^2|z|}\right)\,\mathrm{d}y\,\mathrm{d}z\\ &=\frac{C_d}{16\pi k^4}\int_{0}^\infty\int_{\mathbb{S}^2}\int_{\mathbb{R}^3}(e^
		{-{\rm i}k(\hat x \cdot y-\rho)}-e^{-{\rm i}k\hat x \cdot y-k\rho})\rho V\left(y\right)f_\epsilon\left(y-\rho\omega\right)\,\mathrm{d}y\,\mathrm{d}\sigma(\omega)\,\mathrm{d}\rho.
\end{align*} Denote \begin{align*}
	I_\epsilon(y,\rho)=\rho 1_{[0,\infty)}(\rho)\int_{\mathbb{S}^2}V(y)f_\epsilon\left(y-\rho\omega\right)\,\mathrm{d}\sigma(\omega)
\end{align*} 
and 
\begin{align*}
    T I_\epsilon(y,k)=\int_{\mathbb R}(e^{{\rm i}k\rho}-e^{-k\rho} )I_\epsilon(y,\rho)\,\mathrm{d}\rho.
\end{align*} 
From \cite[Lemma 4.7]{CARO} we have \[
\int_{\mathbb S^2}\left|\int_{\mathbb R^3}e^{-{\rm i}k\hat x\cdot y}T I_\epsilon(y,k)\,\mathrm{d}y\right|^2k^2\,\mathrm{d}\sigma(\hat x) \lesssim \int_{\mathbb R^3}| T I_\epsilon(y,k)|^2\,\mathrm{d}y,
\] which implies  \begin{align}
	&\int_{\mathbb{S}^2}\int_1^\infty s^{m+3+\delta_2} |u^\infty_{1, \epsilon}(\hat x,k)|^2\,\mathrm{d}k\,\mathrm{d}\sigma(\hat x) \notag\\
	&\lesssim \int_1^\infty \int_{\mathbb S^2} k^{m-5+\delta_2}\left|\int_{\mathbb R^3}e^{-{\rm i}k\hat x\cdot y}T I_\epsilon(y,k)\,\mathrm{d}y\right|\,\mathrm{d}\sigma(\hat x)\,\mathrm{d}k\notag\\ &\lesssim \int_{\mathbb R^3}\int_1^\infty k^{m-7+\delta_2}| T I_\epsilon(y,k)|^2\,\mathrm{d}k\,\mathrm{d}y 
	= \int_{\mathbb R^3}\int_1^\infty  k^{ 2 N_m}\left| TI_\epsilon(y,k)\right|^2\,\mathrm{d}k\,\mathrm{d}y, \label{3.12}
\end{align} where \begin{align}\label{delta2}
	N_m= \begin{cases}
		1, &\quad 7 \le m  < 9, \\   0, &\quad 2 < m  < 7,
	\end{cases}\quad \delta_2=\begin{cases}
		9-m, &\quad 7 \le m  < 9, \\   7-m, &\quad 2 < m  < 7.
	\end{cases}
\end{align}
It should be noted that \eqref{delta2} implies a possible larger range of $m$.

Notice that \[
T I_\epsilon(y,k)= \int_{\mathbb R} (e^{{\rm i}k\rho}+e^{-k\rho})I_\epsilon(y,\rho)\,\mathrm{d}\rho= \mathcal{F}^{-1}(I_\epsilon(y,\cdot))(k)+\mathscr{L}( I_\epsilon(y,\cdot))(k),
\]  where $\mathscr{L}(g)$ is the Laplace transform of a function $g$ defined as follows \begin{align*}
    \mathscr{L}(g)(s) := \int_0^\infty g(x)e^{-sx}\,\mathrm{d}x,\quad s>0.
\end{align*}
Then plugging the following identities \begin{align*}
   s\mathcal{F}^{-1}(\partial_x g)(s)={\rm i}\mathcal{F}^{-1}(\partial_x g)(s),\quad\mathscr{L}(g)(s)=g(0)+\mathscr{L}(\partial_x g)(s)
\end{align*} into \eqref{3.12} gives \begin{align*}
   \int_{\mathbb{S}^2}\int_1^\infty k^{m+3+\delta_2} |u^\infty_{1, \epsilon}(\hat x,k)|^2\,\mathrm{d}k\,\mathrm{d}\sigma(\hat x) \lesssim\int_{\mathbb R^3}\int_{\mathbb R}|\partial_\rho^{N_m}I_\epsilon(y,\rho)|^2\,\mathrm{d}\rho\,\mathrm{d}y,
\end{align*} where we have used the Plancherel identity and the boundedness of the Laplace transform in $L^2(\mathbb R^+)$.
 It is easy to verify that the support of $I_\epsilon$ is compact, which implies that there exists $R_0>0$ such that 
\begin{align*}
	&\limsup_{\epsilon \to 0} \mathbb E \int_{\mathbb{S}^2}\int_1^\infty k^{m+3+\delta_2} |u^\infty_{1, \epsilon}(\hat x,k)|^2\,\mathrm{d}k\,\mathrm{d}\sigma(\hat x) \\ &\lesssim \limsup_{\epsilon \to 0} \mathbb E \int_{|y|<R_0} \int_{|\rho|<R_0} |\partial_\rho^{N_m}I_\epsilon(y,\rho)|^2\,\mathrm{d}y \,\mathrm{d}\rho. 
\end{align*} Notice that there holds \begin{align*}
	\partial_\rho I_\epsilon(y,\rho) &=1_{[0,\infty)}(\rho)\int_{\mathbb{S}^2}V(y)f_\epsilon(y-\rho\omega)\,\mathrm{d}\sigma(\omega)\\ &\quad -\rho 1_{[0,\infty)}(\rho)\int_{\mathbb{S}^2}V(y) \nabla f_\epsilon(y-\rho\omega) \cdot \omega\,\mathrm{d}\sigma(\omega)\\ &:= I_{1,\epsilon}(y,\rho)+I_{2,\epsilon}(y,\rho).
\end{align*}
Then we have
 \[\mathbb E [ \partial_\rho I_\epsilon(y,\rho)\overline{\partial_\rho I_\epsilon(y,\rho)}]=\sum_{i,j=1}^2\mathbb{E} [I_{i,\epsilon}(y,\rho){I_{j,\epsilon}(y,\rho)}].\]
Moreover, it holds that \begin{align}
	\lim_{\epsilon \to 0}\mathbb E [f_\epsilon(y-\rho\omega)f_\epsilon(y-\rho\widetilde\omega)] &= K_f\left(y-\rho\omega,y-\rho\widetilde \omega\right), \label{limit1}\\ \lim_{\epsilon \to 0}\mathbb E [\nabla f_\epsilon(y-\rho\omega)\cdot \omega f_\epsilon(y-\rho\widetilde\omega)] &= \nabla_1 K_f\left(y-\rho\omega,y-\rho\widetilde \omega\right)\cdot \omega, \label{limit2}\\ \lim_{\epsilon \to 0}\mathbb E [ f_\epsilon(y-\rho\omega) \nabla f_\epsilon(y-\rho\widetilde\omega)\cdot \widetilde \omega] &= \nabla_2 K_f\left(y-\rho\omega,y-\rho\widetilde \omega\right)\cdot \widetilde\omega, \label{limit3}\\ \lim_{\epsilon \to 0}\mathbb E [\nabla f_\epsilon(y-\rho\omega)\cdot \omega \nabla f_\epsilon(y-\rho\widetilde\omega)\cdot \widetilde \omega] &= \omega^\top\nabla^2_{12} K_f\left(y-\rho\omega,y-\rho\widetilde \omega\right) \omega, \label{limit4}
\end{align}  where $\omega,\widetilde{\omega}\in\mathbb S^2$ and we denote
\[\nabla_1 K_f(x,y):=\nabla_x K_f(x,y), \nabla_2 K_f(x,y):=\nabla_y K_f(x,y),\nabla^2_{12}K(x,y):=(\partial_{x_i}\partial_{y_j}K(x,y))_{ij}.\] Here $K_f(x, y)$ is the kernel of the covariance operator of $f$ introduced in Preliminaries.
From Lemma \ref{Kernel} we know that for $m >2$, $\nabla_x K_f(x,y)$ is weakly singular, and when $m\ge 7$, $\nabla_1 K(x,y),\nabla_2 K_f(x,y),\nabla^2_{12}K(x,y)$ are all weakly singular. Thus, the limit \eqref{limit1} holds in pointwise sense as $m>2$ and the limits \eqref{limit2}--\eqref{limit4} hold in pointwise sense as $m\ge 7$. 
Therefore, we have \begin{align*}
	&\limsup_{\epsilon \to 0} \mathbb E\int_{|y|<R_0} \int_{|\rho|<R_0}|\partial_{\rho}^{N_m}I_{\epsilon}(y,\rho)|^2 \,\mathrm{d}\rho\,\mathrm{d}y\\& \lesssim \int_{|y|<R_0} \int_0^{R_0} \rho^2\int_{\mathbb{S}^2}\int_{\mathbb{S}^2} V(y)^2|K_f(y-\rho\omega, y-\rho\widetilde\omega)|\,\mathrm{d}\sigma(\widetilde{\omega})\,\mathrm{d}\sigma(\omega)\,\mathrm{d}\rho\,\mathrm{d}y \\ & \lesssim \int_{|y|<R_0} \int_0^{R_0} \int_{\mathbb{S}^2}\int_{\mathbb{S}^2} \rho|\omega-\widetilde\omega|^{-1} \,\mathrm{d}\sigma(\widetilde{\omega})\,\mathrm{d}\sigma(\omega)\,\mathrm{d}\rho\,\mathrm{d}y < \infty
\end{align*} for $2<m <7$ with $N_m=0$, and \begin{align*}
	&\limsup_{\epsilon \to 0} \mathbb E\int_{|y|<R_0} \int_{|\rho|<R_0}|\partial_\rho^{N_m} I_{\epsilon}(y,\rho)|^2 \,\mathrm{d}\rho\,\mathrm{d}y\\& \lesssim \int_{|y|<R_0} \int_0^{R_0} \int_{\mathbb{S}^2}\int_{\mathbb{S}^2} V(y)^2|K_f(y-\rho\omega, y-\rho\widetilde\omega)|\,\mathrm{d}\sigma(\widetilde{\omega})\,\mathrm{d}\sigma(\omega)\,\mathrm{d}\rho\,\mathrm{d}y \\ & \quad+ \int_{|y|<R_0} \int_0^{R_0} \rho\int_{\mathbb{S}^2}\int_{\mathbb{S}^2} V(y)^2|\nabla_1 K_f(y-\rho\omega, y-\rho\widetilde\omega) \cdot \omega|\,\mathrm{d}\sigma(\widetilde{\omega})\,\mathrm{d}\sigma(\omega)\,\mathrm{d}\rho\,\mathrm{d}y\\ & \quad+ \int_{|y|<R_0} \int_0^{R_0}\rho \int_{\mathbb{S}^2}\int_{\mathbb{S}^2} V(y)^2|\nabla_2 K_f(y-\rho\omega, y-\rho\widetilde\omega) \cdot \widetilde\omega|\,\mathrm{d}\sigma(\widetilde{\omega})\,\mathrm{d}\sigma(\omega)\,\mathrm{d}\rho\,\mathrm{d}y\\ & \quad+ \int_{|y|<R_0} \int_0^{R_0} \rho^2\int_{\mathbb{S}^2}\int_{\mathbb{S}^2} V(y)^2|\omega^\top\nabla^2_{12} K_f(y-\rho\omega, y-\rho\widetilde\omega)  \widetilde\omega|\,\mathrm{d}\sigma(\widetilde{\omega})\,\mathrm{d}\sigma(\omega)\,\mathrm{d}\rho\,\mathrm{d}y\\ & \lesssim \int_{|y|<R_0} \int_0^{R_0} \int_{\mathbb{S}^2}\int_{\mathbb{S}^2} |\rho(\omega-\widetilde\omega)|^{m-3} +\rho|\rho(\omega-\widetilde\omega)|^{m-4} +\rho^2|\rho(\omega-\widetilde\omega)|^{m-5} \,\mathrm{d}\sigma(\widetilde{\omega})\,\mathrm{d}\sigma(\omega)\,\mathrm{d}\rho\,\mathrm{d}y \\
	&< \infty
\end{align*} for $7\le m <9$ with $N_m=1$. Thus, from the above two estimates, we obtain \eqref{oneorder3e} which gives \eqref{oneorder3} with $\delta_2$ given by \eqref{delta2}.
Combining \eqref{zeroorder}, Lemma \ref{2dim12} and \eqref{oneorder3}, we derive for $2<m<9$ 
\begin{align}\label{Tail3}
   &\int_{\mathbb S^2} |\widehat h(\tau \hat{x})|^2\mathrm{d}\sigma(\hat x)\notag\\
   &=\frac{1}{k}\int_{\mathbb S^2}\int_k^{2k}|C(t,\tau,d)\mathbb E[u^\infty(\hat x,t+\tau)\overline{u^\infty(\hat x,t)}]|^2\,\mathrm{d}t\,\mathrm{d}\sigma(\hat x)+O(k^{-2\widetilde\delta_2})
\end{align} with positive constant $\widetilde\delta_2<\min\{1,\delta_2\}$, where $\delta_2$ is given by \eqref{delta2}. Letting $\tau=\eta k$ in \eqref{Tail3} with $\eta \in (0,1)$ yields the estimate
\begin{align}\label{crucial}
    \int_{\mathbb S^2}|\widehat h(r \hat x)|^2\,\mathrm{d}\sigma(\hat x) \lesssim \sup_{\eta \in (0,1),\, \hat{x} \in \mathbb{S}^2} \frac{1}{k}\int_k^{2k}\Big|k^{m+4}\mathbb E[u^\infty(\hat{x}, t+k\eta) \overline{u^\infty(\hat{x}, t)}]\Big|^2\,\mathrm{d}t + \frac{1}{k^{2\widetilde\delta_2}}
\end{align} for all $r \in (0,k)$ . Based on the above estimate which has a similar form to \eqref{STEP1},
we introduce the data discrepancy in a finite interval of frequencies $I = (K_0, K]$ with $C_0<K_0<K$ as follows
\[
\varepsilon_2^2 =  \sup_{k\in I, \eta\in(0, 1), \hat{x}\in\mathbb S^2} \varepsilon_2^2(k, \eta, \hat{x}),
\]
where
\[
\varepsilon_2^2(k, \eta, \hat{x})=\frac{1}{k}\int_k^{2k}\Big|t^{m+4}[u^\infty(\hat{x}, t+k\eta) \overline{u^\infty(\hat{x}, t)}]\Big|^2\,\mathrm{d}t.
\] As before, we have $\overline{u(x, k)} = u(x,{\rm i}k)$ for $k>0$. Thus, we may meromorphically extend $\varepsilon_2^2(\cdot,\eta, \hat{x})$ from $\mathbb R^+$ to $\mathbb C $ as follows \begin{align*}
	\varepsilon_2^2(k, \eta, \hat{x}) =\frac{1}{k}\int_k^{2k} t^{2(m+4)}u^\infty(\hat{x},t+k\eta) {v^\infty(\hat{x},t)}v^\infty(\hat{x},t+k\eta) {u^\infty(\hat{x},t)}\,\mathrm{d}t,
\end{align*} 
where $v^\infty(\hat{x}, t)$ is defined by \eqref{v}.
The above integral with respect
to the complex variable $t$ is taken over the line segment joining points $k$ and $2k$ in $\mathbb{C}$. The following lemma shows that there exists an infinite slab $\mathcal{R}$ such that $\varepsilon^2_2(\hat{x},\eta, k)$ is analytic and bounded for $k \in \mathcal{R}$. The proof is similar to that of Lemma \ref{bded} and thus we
omit it for brevity.  
\begin{lemma}\label{bded2}
	 There exists an infinite slab $\mathcal{R}$ such that $\varepsilon_2^2(k,\eta,\hat{x})$ is analytic for $k \in \mathcal{R}$ 
	with the following estimate \[
	|\varepsilon_2^2(k, \eta, \hat{x})| \lesssim |k|^{2m}.
	\] 
\end{lemma} 

Now using the estimate \eqref{crucial} and Lemma \ref{bded2}, we obtain the stability
estimate by applying the analytic continuation as in the proof of Theorem \ref{Thm3.1}.
\begin{theorem}\label{Thm3.2}
	Assume that $V \in L^\infty_{\rm comp}(\mathbb R^3)$ and the random source $f$ satisfies \emph{\emph{Assumption (A)}} with $2<m<9$ and $h \in \mathcal{C}_s$. 
	Then the following stability estimate
	holds:
	\begin{align*}
		\|h\|_{L^2(\mathbb{R}^3)}^2
		\lesssim K^{\beta_0}\varepsilon_2^2+\frac{1}{K^\beta(\ln|\ln \varepsilon_2|)^\beta},
	\end{align*}
	where $\beta_0=dt+2m$ and $\beta = \min\{\frac{\widetilde{\delta}_2}{2}-\frac{dt}{2}, st\}$. Here the positive constant $t$ satisfies $0<t<\frac{\widetilde\delta_2}{d}$, where $\widetilde\delta_2<\min\{1,\delta_2\}$ with $\delta_2$ given by \eqref{delta2}.
\end{theorem}

To conclude this section, we discuss the stability by using far-field pattern averaged over the high-frequency band in two dimensions.
A natural idea is to apply the arguments for Theorem \ref{Thm3.2}. However,
since the Green function in two dimensions is $\frac{{\rm i}}{4}H_0^{(1)}(\lambda |x-y|)$ which is more complicated to analyze,
the proposed method for three dimensions may not be applicable. 
    This issue can be fixed  by assuming that the supports of the source $f$ and the potential $V$ are disjoint such that $|x-y|$ has a positive lower bound for $x \in {\rm supp}\,f$ and $y \in {\rm supp}\,V$. Under this assumption, the Hankel function has the asymptotic expansion in terms of an exponential function as follows:
 \begin{align*}
        H_0^{(1)}(\lambda |x-y|)=\sqrt{\frac{1}{z}}e^{{\rm i}(\lambda|x-y|-\frac{\pi}{4})} \sum_{j=0}^N a_j (\lambda|x-y|)^{-j} + \mathcal{O}((\lambda|x-y|)^{-N-1})
    \end{align*} with constants $a_j$.
  Therefore, we can obtain the corresponding stability in two dimensions by following the arguments in three dimensions.
    We point it out that the assumption that $f$ and $V$ have disjoint supports was also used in \cite{llm}.

\section{Conclusion}

This paper considers direct and inverse random source scattering problems for the biharmonic Schr\"odinger operator in two and three dimensions.
For the direct problem, the meromorphic continuation of the resolvent of the biharmonic Schr\"odinger operator is investigated.
Since the random source function may be too rough to exist pointwisely, it should be understood in the sense of distribution and then more efforts
are devoted to analyze the resolvent.
We obtain an analytic domain and estimates for the resolvent. 
 As an application, well-posedness of the direct problem and a Born series expansion for the solution are achieved. Based on the analysis of the resolvent 
and the Born series expansion, we derive increasing stability estimates for the inverse random source problem by using multi-frequency far-field pattern.
A more challenging direction is to study the stability for the inverse random potential scattering problem. Such problem is nonlinear and the present method 
is no applicable.
We hope to report progresses on this subject
in future works.


\begin{thebibliography}{99}
  
  \bibitem{adams}
Robert~A Adams and John~JF Fournier.
\newblock {\em Sobolev {S}paces}.
\newblock Elsevier, 2003.

\bibitem{bao2016random}
Gang Bao, Chuchu Chen, and Peijun Li.
\newblock Inverse random source scattering problems in several dimensions.
\newblock {\em SIAM/ASA Journal on Uncertainty Quantification},
  4(1):1263--1287, 2016.
  
  \bibitem{bao2017random}
Gang Bao, Chuchu Chen, and Peijun Li.
\newblock Inverse random source scattering for elastic waves.
\newblock {\em SIAM Journal on Numerical Analysis}, 55:2616--2643, 2017.

  
  \bibitem{CARO}
Pedro Caro, Tapio Helin, and Matti Lassas.
\newblock Inverse scattering for a random potential.
\newblock {\em Analysis and Applications}, 17(04):513--567, 2019.

  

  
  \bibitem{DZ}
Semyon Dyatlov and Maciej Zworski.
\newblock {\em Mathematical {T}heory of {S}cattering {R}esonances}, volume 200.
\newblock American Mathematical Soc., 2019.



\bibitem{Finco}
  Domenico Finco and Kenji Yajima. 
  \newblock The $ {L}^p $ boundedness of wave operators for {S}chr$ \ddot{\rm o}$dinger operators with threshold singularities {II}. {E}ven dimensional case, \newblock{\em Journal of Mathematical Sciences, The University of Tokyo}, 13 (2006), 277--346.
  
  \bibitem{3}
Watanabe E, Utsunomiya T, and Wang CM. 
  \newblock Hydroelastic analysis of pontoon-type VLFS: a literature survey. 
\newblock{\em Eng Struct},  2004; 6 :245--56.

\bibitem{4}
Evans DV and Porter R. 
  \newblock Penetration of flexural waves through a periodically constrained thin elastic plate in vacuo and floating on water. 
  \newblock{\em J Engrg Math}, 2007;58:317--37.

\bibitem{5}
Farhat M, Guenneau S, and Enoch S. 
  \newblock Ultrabroadband elastic cloaking in thin plates. 
    \newblock{\em Physics Review Letter.} 2009;103:024301.





\bibitem{isakov}
Victor Isakov. 
\newblock Inverse Source Problems, 
\newblock {\em Mathematical Surveys Monogr}. 34, AMS, Providence, RI, 1990.

\bibitem{lassas2008inverse}
Matti Lassas, Lassi P{\"a}iv{\"a}rinta, and Eero Saksman.
\newblock Inverse scattering problem for a two dimensional random potential.
\newblock {\em Communications in mathematical physics}, 279:669--703, 2008.

\bibitem{llm}
Jingzhi Li, Hongyu Liu, and Shiqi Ma,
\newblock Determining a random Schr\"odinger operator: both potential and source are random, 
\newblock {\em Communication in Mathematical Physics} 381 (2021),  527--556.

\bibitem{liang2024stability}
Peijun Li and Ying Liang.
\newblock Stability for inverse source problems of the stochastic Helmholtz
  equation with a white noise.
\newblock {\em SIAM Journal on Applied Mathematics}, 84(2):687--709, 2024.



\bibitem{li2022far}
Jianliang Li, Peijun Li, and Xu Wang.
\newblock Inverse source problems for the stochastic wave equations: far-field
  patterns.
\newblock {\em SIAM Journal on Applied Mathematics}, 82(4):1113--1134, 2022.

\bibitem{LW_1}
Peijun Li and Xu Wang. 
\newblock An inverse random source problem for the biharmonic wave equation. 
\newblock {\em SIAM/ASA Journal. Uncertainty Quantification}, 10 (2022), 949-974.

\bibitem{LW_2}
Peijun Li and Xu Wang. 
\newblock Inverse scattering for the biharmonic wave equation with a random potential, 
\newblock {\em SIAM Journal on Mathematical Analysis}, 56(2024), 1959--1995.

\bibitem{LZZ}
Peijun Li, Jian Zhai and Yue Zhao.
\newblock Stability for the acoustic inverse source problem in inhomogeneous media.
\newblock{\em SIAM Journal on Applied Mathematics}, 80 (2020), 2547--2559.

\bibitem{ZZ}
Jian Zhai and Yue Zhao.
\newblock Increasing stability estimates for the inverse potential scattering problems.
\newblock{\em arXiv preprint arXiv:2306.10211}, 2023

\bibitem{RSWXZ}
Tianjiao Wang, Xiang Xu, and Yue Zhao.
\newblock Stability for a multi-frequency inverse random source problem.
\newblock {\em Inverse Problems}, 40(12):125029, 2024.

\bibitem{WXZ_HN}
Tianjiao Wang, Xiang Xu, and Yue Zhao.
\newblock Inverse random source problems for the Helmholtz and Navier equations in three dimensions.
\newblock {\em preprint}.

  
  

\end{thebibliography}
\end{document}